\documentclass[12pt,a4paper, reqno]{amsart}
\usepackage{amsmath, amssymb, amsthm, eufrak, setspace, verbatim, xypic, enumitem}
\usepackage[top=2.5cm, bottom=2cm, left=2cm, right=2cm]{geometry}
\usepackage[all]{xy}
\usepackage[pagewise]{lineno}

\theoremstyle{definition}
\newtheorem{theorem}{Theorem}[section]

\newtheorem{proposition}[theorem]{Proposition}

\newtheorem{corollary}[theorem]{Corollary}
\newtheorem{example}[theorem]{Example}

\newcommand{\Z}{\mathbb{Z}}
\renewcommand{\bar}{\overline}
\renewcommand{\hat}{\widehat}

\DeclareMathOperator{\Perm}{\mbox{Perm}}
\DeclareMathOperator{\Hol}{\mbox{Hol}}
\DeclareMathOperator{\Aut}{\mbox{Aut}}

\DeclareMathOperator{\Norm}{\mbox{Norm}}

\DeclareMathOperator{\Ima}{\mbox{Im}}

\begin{document}

\title[Sylow's theorem, Cauchy's theorem, Hall's theorem for skew braces]{Analogues of Sylow's first theorem, Cauchy's theorem, and Hall's theorem for skew braces}

\author{Paul J. Truman}
\address{School of Computer Science and Mathematics \\ Keele University \\ Staffordshire \\ ST5 5BG \\ UK}
\email{P.J.Truman@Keele.ac.uk}
\date{\today}



\maketitle

\begin{abstract}
We establish an unconditional analogue of Sylow's first theorem for finite skew braces, and deduce an analogue of Cauchy's theorem. We also prove an analogue of the existence part of Hall's theorem for finite skew braces with soluble additive and multiplicative groups. We make some observations regarding the number of Sylow subskew braces of a skew brace in various cases. By applying these results we streamline the classification of skew braces of order $ pq $, where $ p,q $ are distinct prime numbers. 
\end{abstract}

\section{Introduction}

A \textit{skew brace} is a triple $ (G,\cdot,\circ) $ in which $ (G,\cdot) $ and $ (G,\circ) $ are groups and the operations are connected via a ``twisted" distributivity relation

\begin{equation} \label{eqn_skew_brace}
x \circ (y \cdot z) = (x \circ y) \cdot x^{-1} \cdot (x \circ z) \mbox{ for all } x,y,z \in G,
\end{equation}

\noindent where $ x^{-1} $ denotes the inverse of $ x $ with respect to $ \cdot $. It follows quickly from \eqref{eqn_skew_brace} that the identity elements with respect to $ \cdot $ and $ \circ $ coincide, but the inverse of an element $ x $ with respect to $ \circ $ (denoted $ \bar{x} $) need not coincide with $ x^{-1} $. We suppress the notation $ \cdot $ wherever possible. We will only study skew braces whose underlying set is finite. 

Skew braces were introduced by Guarnieri and Vendramin \cite{GV17}, generalizing Rump's notion of \textit{braces} \cite{Ru07a}. These objects were originally introduced to study bijective nondegenerate solutions of the set-theoretic Yang-Baxter equation, and have been found to have connections with a wide range of other structures and topics; there is therefore great interest in their structure and properties. 

Many concepts from the theory of groups or rings, such as solubility and nilpotency, have skew brace counterparts. In the case of finite skew braces it is certainly natural to study analogues of fruitful approaches from the theory of finite groups. For example: what we may infer from the order of a skew brace about the existence of certain substructures? A subset $ H $ of a skew brace $ (G,\cdot,\circ) $ is called a \textit{subskew brace} if it is a subgroup with respect to both operations. If $ G $ is finite then it is immediate from Lagrange's theorem that the order of a subskew brace divides $ |G| $; we seek conditions on a divisor of $ |G| $ that imply that a subskew brace of that order is guaranteed to exist. Caranti, Del Corso, Ferrara, Di Matteo, and Trombetti \cite{CCMFT25} study analogues of Sylow's first theorem and Hall's theorem for skew braces. Amongst numerous results they show that if $ p^{e} $ is the maximal power of a prime $ p $ that divides $ |G| $ then $ (G,\cdot,\circ) $ contains a subskew brace of order $ p^{e} $ provided $ (G,\cdot,\circ) $ satisfies any one of a range of technical hypotheses, such as being \textit{biskew}  \cite{Ch19b}, \textit{$\gamma$-homomorphic} \cite{BNY22}, \textit{left-nilpotent} \cite{CSV19}, or \textit{two-sided} \cite{Tra23}. From this they deduce a version of Cauchy's theorem for skew braces. This question is also tackled by Damele and P\'{e}rez Calabuig \cite{DPC26}: they prove that if $ (G,\cdot,\circ) $ is a finite skew brace which is biskew or two-sided and $ p $ is a prime number dividing $ |G| $ then $ G $ contains a subskew brace of order $ p $. 

Our aim in this paper is to give unconditional proofs of these results. 
\\ \\
\textbf{Acknowledgements:} I am extremely grateful to Andrea Caranti and Ilaria Del Corso for pointing out how the arguments used to prove Theorem \ref{thm_Sylow} could be adapted to prove Theorem \ref{thm_Hall}, and to Marco Damele for suggesting Corollary \ref{cor_p_r}. I am also grateful to Marco Trombetti for correcting an oversight in Proposition \ref{prop_counting_biskew}, to Ilaria Colazzo and Alan Koch for valuable conversations and feedback, and to the anonymous referees for several improvements to the exposition. 

\section{Holomorphs and $ \gamma $-functions, and main results} \label{sec_main}

If $ G $ is a set and $ \cdot,\circ $ are two binary operations on $ G $, each giving a group structure on $ G $ and with a common identity element, then we may consider the two left regular representations $ \lambda_{\cdot}, \lambda_{\circ} : G \rightarrow \Perm(G) $. The triple $ (G,\cdot,\circ) $ forms a skew brace if and only if $ \lambda_{\circ}(G) $ is contained in the normaliser of $ \lambda_{\cdot}(G) $ in $ \Perm(G) $ \cite{GV17}. This subgroup is called the \textit{(permutational) holomorph} of $ (G,\cdot) $, and is equal to the semidirect product of $ \lambda_{\cdot}(G) $ and $ \Aut(G,\cdot) $. We will often identify it with the external semidirect product $ \Hol(G,\cdot) = (G,\cdot) \rtimes \Aut(G,\cdot) $ (the \textit{abstract} holomorph of $ (G,\cdot) $), which acts on $ G $ from the left:
\begin{equation} \label{eqn_hol_action}
(x,\alpha)[y] = x\alpha(y) \mbox{ for all } (x,\alpha) \in \Hol(G,\cdot) \mbox{ and } y \in G. 
\end{equation}

If $ (G,\cdot,\circ) $ is a skew brace, so that $ \lambda_{\circ}(G) \subseteq \Hol(G,\cdot) $, then for each $ x \in G $ we have $ \lambda_{\circ}(x) = (\mu_{x}, \gamma_{x}) $ for some $ \mu_{x} \in G $ and $ \gamma_{x} \in \Aut(G,\cdot) $. Evaluating at the identity element of $ G $ shows quickly that $ \mu_{x} = x $, and projecting onto the automorphism component yields a homomorphism $ \gamma : (G,\circ) \rightarrow \Aut(G,\cdot) $ which we call the \textit{$ \gamma $-function} of the skew brace. This function is often denoted $ \lambda $ in the literature; we reserve this symbol for left regular representations. 

The $ \gamma $-function of a skew brace $ (G,\cdot,\circ) $ translates between the binary operations: we have

\begin{equation} \label{eqn_gamma_translates}
x \circ y = x \gamma_{x}(y) \mbox{ for all } x,y \in G. 
\end{equation}

It can also be used to characterise various classes of substructures. A subset $ H $ of $ G $ is a subskew brace if and only if it is a subgroup with respect to one of the operations with the additional property that $ \gamma_{x}(y) \in H $ for all $ x,y \in H $. A subskew brace $ H $ is called a \textit{left ideal} if it satisfies the stronger property $ \gamma_{x}(y) \in H $ for all $ x \in G $ and $ y \in H $. Finally, a left ideal $ H $ is called an  \textit{ideal} if it is normal with respect to both operations; ideals are the kernels of skew brace homomorphisms and the substructures that permit the formation of quotients.

With these definitions to hand we state and prove our first result.  

\begin{theorem} \label{thm_Sylow}
Let $ G=(G,\cdot,\circ) $ be a finite skew brace, let $ p $ be a prime number, and write $ |G|=p^{e}m $ with $ p \nmid m $. Then $ G $ contains a subskew brace of order $ p^{e} $. 
\end{theorem}
\begin{proof}
Let $ X $ denote the set of Sylow $ p $-subgroups of $ (G,\cdot) $. Since $ \gamma_{x} \in \Aut(G,\cdot) $ for each $ x \in G $, and $ \gamma : (G,\circ) \rightarrow \Aut(G,\cdot) $ is a homomorphism, the group $ (G,\circ) $ acts on $ X $ via $ \gamma $. 
\\ \\
Now let $ (Q,\circ) $ be a Sylow $ p $-subgroup of $ (G,\circ) $; then $ (Q,\circ) $ also acts on $ X $ via $ \gamma $. Since $ (Q,\circ) $ is a $ p $-group, and $ |X| \equiv 1 \pmod{p} $, there must be an orbit of length $ 1 $, say $ \{ P \} $. Thus there exists a Sylow $ p $-subgroup $ (P,\cdot) $ of $ (G,\cdot) $ such that $ \gamma_{x}(P)=P $ for all $ x \in Q $. 
\\ \\
Now consider the subgroup $ (G,\cdot) \rtimes \gamma(G) $ of $ \Hol(G,\cdot) $. Since $ \gamma_{x}(P)=P $ for all $ x \in Q $, we can construct a subgroup $ (P,\cdot) \rtimes \gamma(Q) $ of $ (G,\cdot) \rtimes \gamma(G) $. Since $ (Q,\circ) $ is a Sylow $ p $-subgroup of  $ (G,\circ) $ its image $ \gamma(Q) $ is a Sylow $ p $-subgroup of $ \gamma(G) $, so $ (P,\cdot) \rtimes \gamma(Q) $ is a Sylow $ p $-subgroup of $ (G,\cdot) \rtimes \gamma(G) $. In particular, every Sylow $ p $-subgroup of $ (G,\cdot) \rtimes \gamma(G) $ is conjugate to $ (P,\cdot) \rtimes \gamma(Q) $.
\\ \\
Let $ \lambda_{\circ} : (G,\circ) \rightarrow \Hol(G,\cdot) $ be the left regular representation with respect to $ \circ $. Then $ \lambda_{\circ}(G) \subseteq (G,\cdot) \rtimes \gamma(G) $. Consider $ \lambda_{\circ}(Q) $; this is a $ p $-subgroup of $ (G,\cdot) \rtimes \gamma(G) $, so it is contained in a Sylow $ p $-subgroup of $ (G,\cdot) \rtimes \gamma(G) $. Hence there exist $ g,h \in G $ such that 
\[ \lambda_{\circ}(Q) \subseteq (g,\gamma_{h})\left( (P,\cdot) \rtimes \gamma(Q) \right) (g,\gamma_{h})^{-1}, \]
and so
\[ \lambda_{\circ}(Q) (g,\gamma_{h}) \subseteq (g,\gamma_{h})\left( (P,\cdot) \rtimes \gamma(Q) \right). \] 
Evaluating both sides at the identity element of $ G $ we have 
\[ Q \circ g \subseteq g \gamma_{h}(P), \]
and so 
\begin{align*}
\bar{g} \circ Q \circ g & \subseteq \bar{g} \circ (g\gamma_{h}(P)) \\
& = (\bar{g} \circ g) \bar{g}^{-1} (\bar{g} \circ \gamma_{h}(P)) \\
& = \bar{g}^{-1} (\bar{g} \circ \gamma_{h}(P)) \\
& = \gamma_{\bar{g}\circ h}(P).
\end{align*}
But $ \bar{g} \circ Q \circ g $ is a Sylow $ p $-subgroup of $ (G,\circ) $, and $ \gamma_{\bar{g} \circ h}(P) $ is a Sylow $ p $-subgroup of $ (G,\cdot) $ (since $ \gamma_{\bar{g} \circ h} \in \Aut(G,\cdot) $). Hence we have
\[ \bar{g} \circ Q \circ g  = \gamma_{\bar{g} \circ h}(P), \]
and this set is a Sylow $ p $-subgroup with respect to both operations simultaneously. That is: a subskew brace of order $ p^{e} $. 
\end{proof}

The following corollary makes use of the $ \ast $ operation on a skew brace $ S $. We summarise the relevant properties, as described in \cite[Section 2]{CSV19}. For $ x,y \in S $ we define $ x \ast y = \gamma_{x}(y)y^{-1} $, and for $ X,Y \subseteq S $ we define $ X \ast Y $ to be the additive subgroup generated by $ \{ x \ast y \mid x \in X, \; y \in Y \} $. In particular, we define a chain of subskew braces (in fact, left ideals) $ S^{n} $ by $ S^{1}=S $ and $ S^{n+1} = S \ast S^{n} $ for $ n \geq 1 $. In particular, $ S^{2} $ is an ideal of $ S $ and $ S / S^{2} $ is trivial as a skew brace (that is: the two operations involved coincide). We say that $ S $ is \textit{left nilpotent} if $ S^{k} = \{ 1 \} $ for some $ k \in \mathbb{N} $. In particular, if $ S $ is left nilpotent then $ S^{2} $ is properly contained in $ S $. 

\begin{corollary} \label{cor_p_r}
With the notation above, $ G $ contains a subskew brace of each order $ p^{r} $ with $ 0 \leq r \leq e $. 
\end{corollary}
\begin{proof}
Let $ S $ be a subskew brace of $ G $ of order $ p^{e} $. By \cite[Proposition 4.4]{CSV19} $ S $ is left nilpotent, and so $ S^{2} $ is a proper ideal of $ S $. Now $ S / S^{2} $ is a trivial skew brace of $ p $-power order greater than $ 1 $. That is: a $ p $-group of order greater than $ 1 $. Hence $ S / S^{2} $ contains a subgroup of index $ p $; viewing $ S / S^{2} $ as a trivial skew brace, this subgroup is an ideal, which therefore lifts to an ideal of $ S $ of index $ p $. Hence we obtain a chain of subskew braces of $ G $
\[ S = S_{e} \supset S_{e-1} \supset \cdots \supset S_{1} \supset S_{0} = \{ 1 \} \]
in which each $ S_{i-1} $ has index $ p $ in $ S_{i} $. Therefore $ G $ contains a subskew brace of each order $ p^{r} $ with $ 0 \leq r \leq e $. 
\end{proof}

In particular, choosing $ r=1 $ in Corollary \ref{cor_p_r} we obtain a skew brace analogue of Cauchy's theorem. 

\begin{corollary} \label{cor_Cauchy}
Let $ (G,\cdot,\circ) $ be a finite skew brace and let $ p $ be a prime number dividing $ |G| $. Then $ G $ contains a subskew brace of order $ p $. 
\end{corollary}

It is natural to ask for skew brace analogues of the remaining Sylow theorems. For example, do the Sylow subskew braces of a finite skew brace $ G $ form a single orbit under a suitable action of some group closely connected with $ G $?  Very small examples already illustrate that for a fixed prime $ p $ dividing $ |G| $ the set of Sylow $ p $-subskew braces of $ G $ is not stable under the obvious actions of $ (G,\cdot) $ or $ (G,\circ) $ via conjugation, or under the natural action of $ (G,\circ) $ via $ \gamma $. 

\begin{example} \label{eg_Z_6_counterexamples}
Define a binary operation $ \circ $ on $ \Z_{6} $ by $ i \circ j = i+(-1)^{i}j $. Then $ (\Z_{6},\circ) \cong D_{3} $ and $ (\Z_{6},+,\circ) $ is a biskew brace (that is: $ (\Z_{6},\circ,+) $ is also a skew brace). 

Obviously $ (\Z_{6},+) $ has a unique Sylow $ 2 $-subgroup (generated by $ 3 $), whereas $ (\Z_{6},\circ) $ has three (generated by $ 1,3 $, and $ 5 $). Hence $ (\Z_{6},+,\circ) $ has a unique Sylow $ 2 $-subskew brace (which is a left ideal), but this is not stable under conjugation by the multiplicative group $ (\Z_{6},\circ) $. 

Reversing the roles of the operations we see that $ (\Z_{6},\circ,+) $ has a unique Sylow $ 2 $-subskew brace, but this is not stable under conjugation by the additive group $ (\Z_{6},\circ) $. Moreover, the $ \gamma $-function of $ (\Z_{6},\circ,+) $ is given by $ \gamma_{i}(j) = (-1)^{i}j $, so the unique Sylow $ 2 $-subskew brace is not stable under the action of $ (\Z_{6},\circ) $ via $ \gamma $. 
\end{example}


Without a description of the Sylow $p$-subskew braces of a finite skew brace $G$ as a single orbit under some suitable group action, we obviously cannot hope to imitate classical results to derive information about their number. We record some observations for skew braces satisfying various additional hypotheses.

Certainly if $ (G,\cdot) $ has a unique Sylow $ p $-subgroup (in particular, if it is nilpotent) then $ G $ has a unique Sylow $ p $-subskew brace, which is a left ideal. Similarly, if $ (G,\circ) $ has a unique Sylow $ p $-subgroup then $ G $ has a unique Sylow subskew brace, although this is not necessarily an ideal, even if $ (G,\circ) $ is nilpotent: consider the skew brace $ (\Z_{6},\circ,+) $ in Example \ref{eg_Z_6_counterexamples}.

On the other hand, if $ G $ is a two-sided skew brace then we observe the following behaviour at the other extreme (see also \cite[Theorem 14]{CCMFT25}):

\begin{proposition}
Suppose that $ G $ is a finite two-sided skew brace, and let $ p $ be a prime number. Then every Sylow $ p $-subgroup of $ (G,\circ) $ is a subskew brace of $ G $. Consequently, the Sylow $ p $-subskew braces of $ G $ are mutually conjugate in $ (G,\circ) $, and the number of these subskew braces divides $ |G| $ and is congruent to $ 1 $ modulo $ p $. 
\end{proposition}
\begin{proof}
Let $ P $ be a Sylow $ p $-subskew brace of $ G $. Then the set of Sylow $ p $-subgroups of $ (G,\circ) $ is equal to the set of conjugates of $ (P,\circ) $ in $ (G,\circ) $. But since $ G $ is two-sided every inner automorphism of $ (G,\circ) $ is a skew brace automorphism of $ G $ (\cite[Lemma 4.1]{Na19},  \cite[Proposition 2.3]{Tra23}), so for each $ g \in G $ the set $ g \circ P \circ \bar{g} $ is closed with respect to $ \cdot $, so is a subskew brace of $ G $. The other claims follow immediately. 
\end{proof}

In \cite[Lemma 3.7]{Ca20} and \cite[Lemma 2.9]{ST23b} three properties of a skew braces are studied: being biskew, being $ \gamma $-homomorphic, and satisfying $ \Ima(\gamma) $ abelian. Any two of these properties implies the third. In this case we can obtain a concrete description of the Sylow $p$-subskew braces by studying the set of Sylow $ p $-subgroups of $ (G,\cdot) $.

\begin{proposition} \label{prop_counting_biskew}
Suppose that $ G $ is a finite $ \gamma $-homomorphic skew brace such that $ \Ima(\gamma) $ is abelian. Let $ p $ be a prime number, let $ X $ be the set of Sylow $ p $-subgroups of $ (G,\cdot) $, and let $ (P,\cdot) \in X $. Then the Sylow $ p $-subskew braces correspond with the orbits of length $ 1 $ when $ (P,\cdot) $ acts on $ X $ via $ \gamma $. 
\end{proposition}
\begin{proof}
Clearly $ X $ is equal to the set of conjugates of $ (P,\cdot) $ in $ (G,\cdot) $. For $ g \in G $, the subgroup $ gPg^{-1} $ of $ (G,\cdot) $ is a subskew brace if and only if $ \gamma_{gxg^{-1}}(gPg^{-1}) = gPg^{-1} $ for all $ x \in P $. Since $ G $ is $ \gamma $-homomorphic and $ \Ima(\gamma) $ is abelian we obtain $ \gamma_{gxg^{-1}} = \gamma_{x} $ for all $ x \in P $ and $ g \in G $, and so $ gPg^{-1} $ is a subskew brace if and only if $ \gamma_{x}(gPg^{-1}) = gPg^{-1} $ for all $ x \in P $. Therefore the Sylow $ p $-subskew braces of $ G $ correspond with the orbits of length $ 1 $ when $ (P,\cdot) $ acts on $ X $ via $ \gamma $. 
\end{proof}

\begin{example}
A source of skew braces satisfying the hypotheses of Proposition \ref{prop_counting_biskew} is the theory of \textit{abelian maps}, due to Koch \cite{Koc21a}. Given a group $ (G,\cdot) $ and an endomorphism of $ (G,\cdot) $ with abelian image, the binary operation defined by $ g \circ h = g\psi(g)^{-1}h\psi(g) $ makes $ (G,\cdot,\circ) $ into a biskew brace whose $ \gamma $-function is given by $ \gamma_{g}(h) = \psi(g)^{-1}h\psi(g) $ for all $ g,h \in G $. Since $ \psi $ is abelian we find that $ \Ima(\gamma) $ is abelian; hence $ (G,\cdot,\circ) $ is also $ \gamma $-homomorphic. 

Now suppose that $ G $ is finite and fix a Sylow $ p $-subgroup $ (P,\cdot) $ of $ (G,\cdot) $. By Proposition \ref{prop_counting_biskew}, a conjugate $ gPg^{-1} $ of $ P $ is a Sylow subskew brace of $ (G,\cdot,\circ) $ if and only if $ \psi(x)^{-1}gPg\psi(x) = gPg^{-1} $ for all $ x \in P $. That is, if and only if $ \psi(x) \in \Norm_{(G,\cdot)}(gPg^{-1},\cdot) $ for all $ x \in P $. Since such $ \psi(x) $ have $ p $-power order, this occurs if and only if $ \psi(x) \in gPg^{-1} $ for all $ x \in P $; that is, if and only if $ \psi(P) \subseteq gPg^{-1} $. 

We note that since $ (\psi(G),\cdot) $ is abelian it has a unique Sylow $ p $-subgroup, which is the image under $ \psi $ of every Sylow $ p $-subgroup of $ (G,\cdot) $; hence the condition derived above does not depend upon the choice of $ P $. 
\end{example}

To close this section we specialise to a finite skew brace $(G,\cdot,\circ) $ in which $ (G,\cdot) $ and $ (G,\circ) $ are soluble groups. In this case, a variant of the proof of Theorem \ref{thm_Sylow} yields the following analogue of Hall's theorem:

\begin{theorem} \label{thm_Hall}
Suppose that $ G=(G,\cdot,\circ) $ is a finite skew brace in which $ (G,\cdot) $ and $ (G,\circ) $ are soluble groups. Let $ \pi $ be a set of primes. Then $ G $ contains a subskew brace whose order is a product of primes in $ \pi $ and whose index is not divisible by any of the primes in $ \pi $. 
\end{theorem}
\begin{proof}
Let $ (Q,\circ) $ be a Hall $ \pi $-subgroup of $ (G,\circ) $. By analogy with the proof of Theorem \ref{thm_Sylow}, we seek a Hall $ \pi $-subgroup $ (P,\cdot) $ of $ (G,\cdot) $ such that $ \gamma_{x}(P)=P $ for all $ x \in Q $.

To do this, we consider once again the subgroup $ (G,\cdot) \rtimes \gamma(G) $ of $ \Hol(G,\cdot) $. Since $ (G,\cdot) $ and $ (G,\circ) $ are soluble, this is also soluble. Since $ (Q,\circ) $ is a Hall $ \pi $-subgroup of $ (G,\circ) $ its image $ \gamma(Q) $ is a Hall $ \pi $-subgroup of $ \gamma(G) $, and so $ (1,\gamma(Q)) $ is a $ \pi $-subgroup of $ (G,\cdot) \rtimes \gamma(G) $. Let $ H $ be a Hall $ \pi $-subgroup of $ (G,\cdot) \rtimes \gamma(G) $ that contains $ (1,\gamma(Q)) $, and let $ (P,\cdot) $ be the subgroup of $ (G,\cdot) $ such that $ (P,1) = (G,1) \cap H $. Since $ (G,1) $ is normal in $ (G,\cdot) \rtimes \gamma(G) $ we find that $ (P,\cdot) $ is a Hall $ \pi $-subgroup of $ (G,\cdot) $. Now for $ x \in Q $ we have
\[ (1,\gamma_{x})(P,1)(1,\gamma_{x})^{-1} = (\gamma_{x}(P),1). \]
Obviously we have $ (\gamma_{x}(P),1) \subseteq (G,1) $; in addition, since $ (1,\gamma(Q)) \subseteq H $ we have $ (\gamma_{x}(P),1) \subseteq H $. Hence $ (\gamma_{x}(P),1) \subseteq (G,1) \cap H = (P,1) $, and so $ \gamma_{x}(P) = P $ for all $ x \in Q $. 

Now we may essentially follow the second half of the proof of Theorem \ref{thm_Sylow}. 

We may construct the subgroup $ (P,\cdot) \rtimes \gamma(Q) $, which is a Hall $ \pi $-subgroup of $ (G,\cdot) \rtimes \gamma(G) $; since $ (G,\cdot) \rtimes \gamma(G) $  is soluble, every Hall $ \pi $-subgroup of $ (G,\cdot) \rtimes \gamma(G) $ is conjugate to $ (P,\cdot) \rtimes \gamma(Q) $. 

The subgroup $ \lambda_{\circ}(Q) $ is a $ \pi $-subgroup of $ (G,\cdot) \rtimes \gamma(G) $, so is contained in some conjugate of $ (P,\cdot) \rtimes \gamma(Q) $. As before, we obtain
\[ \lambda_{\circ}(Q) (g,\gamma_{h}) \subseteq (g,\gamma_{h})\left( (P,\cdot) \rtimes \gamma(Q) \right) \] 
for some $ g,h \in G $. Evaluating both sides at the identity element of $ G $ yields 
\[ Q \circ g \subseteq g \gamma_{h}(P), \]
and we find that 
\[ \bar{g} \circ Q \circ g  = \gamma_{\bar{g} \circ h}(P). \]
This set is then a Hall $ \pi $-subgroup with respect to both operations simultaneously. That is: a subskew brace whose order is a product of primes in $ \pi $ and whose index is not divisible by any of the primes in $ \pi $. 
\end{proof}

We remark that the approach employed in the first half of this proof could also be used to give another proof of Theorem \ref{thm_Sylow}.

\section{An application: skew braces of order $ pq $ revisited}

The classical Sylow, Cauchy, and Hall theorems are fundamental tools for analysing and classifying finite groups, and we expect their skew brace analogues to be similarly useful in the study of finite skew braces. As a first example, we consider the question of classifying skew braces of order $ pq $, where $ p > q $ are prime numbers. This classification is due to Acri and Bonatto \cite{AcB20}, employing results of Byott \cite{By04c} concerning regular subgroups of holomorphs of groups of order $ pq $. We will show how the results of Section \ref{sec_main} can be used to simplify this classification. 
\\ \\
If $ p \not \equiv 1 \pmod{q} $ then up to isomorphism there is a unique skew brace of order $ pq $, which is the trivial skew brace on the cyclic group of order $ pq $ \cite[Theorem A.8]{SV18}. 

For the remainder of this section we shall study skew braces $ G=(G,\cdot,\circ) $ of order $ pq $ where $ p \equiv 1 \pmod{q} $. Up to isomorphism there are two groups of order $ pq $: one cyclic, the other metacyclic. Since each of these has a unique subgroup of order $ p $, the skew brace $ G $  contains an ideal $ P $ of order $ p $; this is trivial as a skew brace, and we may write $ P = \langle s \rangle $ unambiguously. By Theorem \ref{cor_Cauchy} the skew brace $ G $ also contains a subskew brace of order $ q $, say $ Q = \langle t \rangle $. Then $ (G,\cdot) $ and $ (G,\circ) $ are both generated by $ s $ and $ t $; we shall write

\begin{equation} \label{eqn_additive_presentaion}
(G,\cdot) = \langle s,t \mid s^{p} = t^{q} = 1, \; tst^{-1} = s^{g} \rangle
\end{equation}
where either $ g=1 $ (corresponding to $ (G,\cdot) $ cyclic) or $ g $ is a fixed element of $ \Z_{p}^{\times} $ of order $ q $ (corresponding to $ (G,\cdot) $ metacyclic). 

An important division in our classification is between those skew braces in which $ Q $ is a left ideal and those in which it is not. Since $ \Aut(P) $ is abelian, a well known construction (\cite[Proposition 4.6.12]{NZ18} or \cite[Proposition 1.1]{AcB20}) implies that for each $ d \in \Z_{p}^{\times} $ such that $ d^{q}=1 $ the binary operation defined by 
\begin{equation} \label{eqn_sdp_skew_brace}
s^{i}t^{u} \circ_{d} s^{j}t^{v} = s^{i+d^{u}j}t^{u+v} 
\end{equation}
yields a skew brace $ (G,\cdot,\circ_{d}) $. Since $ s \circ_{d} t = st $, we have $ \gamma_{s}(t)=t $, and so the subskew brace $ Q $ is a left ideal of $ (G,\cdot,\circ_{d}) $. The converse is also true: if  $ (G,\cdot,\circ) $ is a skew brace in which $ Q $ is a left ideal then $ (G,\cdot,\circ) $ is isomorphic to $ (G,\cdot,\circ_{d}) $ for some $ d $ \cite[Proposition 3.3]{Tr26}. 

This construction is sufficient to classify skew braces $ G=(G,\cdot,\circ) $ with $ (G,\cdot) $ cyclic of order $ pq $. This result corresponds with \cite[Main Theorem, first bullet point]{AcB20}.

\begin{proposition} \label{prop_classify_pq_cyclic}
Suppose that $ p,q $ are primes with $ p \equiv 1 \pmod{q} $. Then up to isomorphism there are two skew braces $ G=(G,\cdot,\circ) $ with $ (G,\cdot) $ cyclic of order $ pq $. 
\end{proposition}
\begin{proof}
Suppose that $ (G,\cdot,\circ) $ is a skew brace with $ (G,\cdot) $ cyclic of order $ pq $. Then the subskew brace $ Q $ is a left ideal of $ G $, and so $ G \cong (G,\cdot,\circ_{d}) $ for some $ d $. If $ d=1 $ then we obtain the trivial skew brace on $ (G,\cdot) $. For each $ d \neq 1 $ the operation $ \circ_{d} $ gives a skew brace in which $ (G,\circ_{d}) $ is metacyclic; we find that these skew braces are mutually isomorphic via the automorphisms of $ (G,\cdot) $ defined by $ s \mapsto s, \; t \mapsto t^{u} $ for $ 1 \leq u \leq q-1 $. Hence there are precisely two isomorphically distinct skew braces in this case. 
\end{proof}

Now we turn to skew braces $ G=(G,\cdot,\circ) $ with $ (G,\cdot) $ metacyclic of order $ pq $ (hence we take $ g $ to be a fixed element of order $ q $ in $ \Z_{p}^{\times} $ in \eqref{eqn_additive_presentaion}). In this case we shall combine the construction above with the concept of the \textit{opposite} of a skew brace $ G=(G,\cdot,\circ) $ \cite{KT20}: this is the skew brace $ \hat{G} = (G,\hat{\cdot},\circ) $, where $ (G,\hat{\cdot} \;) $ is simply the opposite group to $ (G,\cdot) $. If $ \gamma $ denotes the $ \gamma $-function of $ G $ then the $ \gamma $-function of $ \hat{G} $ is given by $ x\gamma_{x}(y)x^{-1} $. If $ G' $ is a further skew brace and $ G \cong G' $ then we also have $ \hat{G} \cong \hat{G'} $. 

We note that the proof of the following key result depends heavily on the facts that $ \gamma_{s}(s)=s $ and $ \gamma_{t}(t)=t $, which follow from the fact that $ G $ is guaranteed to have an ideal $ P $ of order $ p $ and a subskew brace $ Q $ of order $ q $. 

\begin{proposition} \label{prop_ker_gamma}
Suppose that $ p,q $ are primes with $ p \equiv 1 \pmod{q} $ and that $ G = (G,\cdot,\circ) $ is a skew brace with $ (G,\cdot) $ metacyclic of order $ pq $. Then exactly one of the inclusions $ P \subseteq \ker(\gamma) $ or $ P \subseteq \ker(\hat{\gamma}) $ holds. 
\end{proposition}
\begin{proof}
First suppose that $ P \subseteq \ker(\gamma) $. Then in particular $ \gamma_{s}(t) = t $, and so $ \hat{\gamma}_{s}(t) = sts^{-1} = s^{1-g}t $. Hence $ P \not \subseteq \ker(\hat{\gamma}) $. 

Conversely, suppose that $ P \not \subseteq \ker(\gamma) $. Since $ P $ is a trivial skew brace we certainly have $ \gamma_{s}(s)=s $, so we must have $ \gamma_{s}(t) = s^{a}t $ for some nonzero $ a \in \Z_{p} $. We claim that $ a=g-1 $, so that $ \gamma_{s}(t) = s^{-1}ts $. 

To establish the claim, write $ t \circ s \circ t^{-1} = s^{d} $ with $ d^{q} = 1 $ in $ \Z_{p}^{\times} $ (possibly $ d=1 $). Then $ \gamma_{t}\gamma_{s}\gamma_{t}^{-1} = \gamma_{s}^{d} $, so 
\begin{align*}
& \gamma_{t}\gamma_{s}\gamma_{t}^{-1}(t) = \gamma_{s}^{d}(t) \\
\Rightarrow \; & \gamma_{t}\gamma_{s}(t) = \gamma_{s}^{d}(t) \tag{$\gamma_{t}(t)=t$} \\
\Rightarrow \; & \gamma_{t}(s^{a}t) = s^{da}t \\
\Rightarrow \; & \gamma_{t}(s)=s^{d} \tag{$ a $ is nonzero in $ \Z_{p} $}
\end{align*}
Now consider again the relation $ t \circ s = s^{d} \circ t $. We have
\[ t \circ s = t \gamma_{t}(s) = ts^{d} = s^{dg}t, \]
and
\[ s^{d} \circ t = s^{d} \gamma_{s}^{d}(t) = s^{d(a+1)}t. \]
Hence $ a=g-1 $, and so $ \gamma_{s}(t) = s^{-1}ts $, as claimed. Hence $ \hat{\gamma}_{s}(t) = s\gamma_{s}(t)s^{-1} = t $. Since we certainly have $ \hat{\gamma}_{s}(s) = s $, we conclude that $ P \subseteq \ker(\hat{\gamma}) $.  
\end{proof}

Using this result we now classify the skew braces $ G $ with $ (G,\cdot) $ metacyclic of order $ pq $. These are the skew braces described in \cite[Main Theorem, second bullet point]{AcB20}

\begin{proposition} \label{prop_classify_pq_metacyclic}
Suppose that $ p,q $ are primes with $ p \equiv 1 \pmod{q} $. Then up to isomorphism there are $ 2q $ skew braces $ G=(G,\cdot,\circ) $ with $ (G,\cdot) $ metacyclic of order $ pq $. 
\end{proposition}
\begin{proof}
Suppose that $ (G,\cdot,\circ) $ is a skew brace of order $ pq $ with $ (G,\cdot) $ metacyclic. By Proposition \ref{prop_ker_gamma} we have $ P \subseteq \ker(\gamma) $ or $ P \subseteq \ker(\hat{\gamma}) $; no skew brace of the first kind can be isomorphic to any skew brace of the second kind, so it is sufficient to focus on the case in which $ P \subseteq \ker(\gamma) $. In this case the subskew brace $ Q $ is a left ideal of $ G $, and so $ G \cong (G,\cdot,\circ_{d}) $ for some $ d $. 

If $ d=1 $ then we obtain a skew brace in which $ (G,\circ) $ is cyclic. Taking the opposite of this skew brace yields one further isomorphically distinct skew brace in which $ (G,\circ) $ is cyclic.

For each $ d \neq 1 $ the operation $ \circ_{d} $ gives a skew brace in which $ (G,\circ_{d}) $ is metacyclic; we find that each automorphism of $ (G,\cdot) $ extends to an automorphism of $ (G,\cdot,\circ_{d}) $, so the skew braces obtained in this way are mutually nonisomorphic. Hence we obtain $ 2(q-1) $ isomorphically distinct skew braces in which $ (G,\circ) $ is metacyclic.

In total, we find that up to isomorphism there are $ 2q $ skew braces $ G=(G,\cdot,\circ) $ with $ (G,\cdot) $ metacyclic. 
\end{proof}

Combining Propositions \ref{prop_classify_pq_cyclic} and \ref{prop_classify_pq_metacyclic} we recover \cite[Main Theorem]{AcB20}: if $ p,q $ are primes with $ p \equiv 1 \pmod{q} $ then there are $ 2q+2 $ isomorphically distinct skew braces of order $ pq $. 

\bibliographystyle{plain}

\end{document}